\documentclass[11pt,letter]{amsart}
\usepackage{geometry}
\usepackage{amsmath,amsthm,amssymb}
\usepackage{comment}
\usepackage{url}
\usepackage[dvipsnames]{xcolor}
\usepackage[normalem]{ulem}
\usepackage[english=american]{csquotes}
\usepackage{enumerate}
\usepackage{textcomp}
\usepackage{mathrsfs}
\usepackage{mathtools}
\usepackage{color, colortbl}
\definecolor{Gray}{gray}{0.9}
\usepackage{bbm}
\usepackage[multiple]{footmisc}
\usepackage{hyperref}
\usepackage{stmaryrd}
\hypersetup{
	bookmarks=true,         % show bookmarks bar?
	unicode=false,          % non-Latin characters in Acrobats bookmarks
	pdftoolbar=true,        % show Acrobats toolbar?
	pdfmenubar=true,        % show Acrobats menu?
	pdffitwindow=false,     % window fit to page when opened
	pdfstartview={FitH},    % fits the width of the page to the window
	pdftitle={My title},    % title
	pdfauthor={Author},     % author
	pdfsubject={Subject},   % subject of the document
	pdfcreator={Creator},   % creator of the document
	pdfproducer={Producer}, % producer of the document
	pdfkeywords={keyword1, key2, key3}, % list of keywords
	pdfnewwindow=true,      % links in new PDF window
	colorlinks=true,       % false: boxed links; true: colored links
	linkcolor=blue ,          % color of internal links (change box color with linkbordercolor)
	citecolor=blue ,        % color of links to bibliography
	filecolor=magenta,      % color of file links
	urlcolor=Aquamarine           % color of external links
}
\usepackage{caption}
\usepackage{amsfonts}
\usepackage{amssymb}
\usepackage{tikz}

\usepackage{siunitx}
\usepackage{xcolor}
\usepackage{pgfplotstable}
\pgfplotsset{compat=1.8}

\definecolor{rulecolor}{RGB}{0,71,171}
\definecolor{tableheadcolor}{gray}{0.92}

\colorlet{tableheadcolor}{gray!25} % Table header colour = 25% gray
\colorlet{tablerowcolor}{gray!10} % Table row separator colour = 10% gray
\makeatletter
\DeclareFontFamily{OMX}{MnSymbolE}{}
\DeclareSymbolFont{MnLargeSymbols}{OMX}{MnSymbolE}{m}{n}
\SetSymbolFont{MnLargeSymbols}{bold}{OMX}{MnSymbolE}{b}{n}
\DeclareFontShape{OMX}{MnSymbolE}{m}{n}{
	<-6>  MnSymbolE5
	<6-7>  MnSymbolE6
	<7-8>  MnSymbolE7
	<8-9>  MnSymbolE8
	<9-10> MnSymbolE9
	<10-12> MnSymbolE10
	<12->   MnSymbolE12
}{}
\DeclareFontShape{OMX}{MnSymbolE}{b}{n}{
	<-6>  MnSymbolE-Bold5
	<6-7>  MnSymbolE-Bold6
	<7-8>  MnSymbolE-Bold7
	<8-9>  MnSymbolE-Bold8
	<9-10> MnSymbolE-Bold9
	<10-12> MnSymbolE-Bold10
	<12->   MnSymbolE-Bold12
}{}
\let\llangle\@undefined
\let\rrangle\@undefined
\DeclareMathDelimiter{\llangle}{\mathopen}%
{MnLargeSymbols}{'164}{MnLargeSymbols}{'164}
\DeclareMathDelimiter{\rrangle}{\mathclose}%
{MnLargeSymbols}{'171}{MnLargeSymbols}{'171}
\makeatother

\numberwithin{equation}{section}

\newtheorem{theorem}{Theorem}[section]
\newtheorem{lemma}[theorem]{Lemma}
\newtheorem{proposition}[theorem]{Proposition}

\theoremstyle{definition}

\newtheorem{remark}[theorem]{Remark}
\newtheorem{example}[theorem]{Example}

\newcommand{\Acal}{\mathcal{A}}

\newcommand{\Ucal}{\mathcal{U}}

\newcommand{\Ffrak}{\mathfrak{F}}

\newcommand{\sW}{\mathscr{W}}
\newcommand{\Wscr}{\mathscr{W}}

\newcommand{\Bbf}{\mathbf{B}}
\newcommand{\Cbf}{\mathbf{C}}

\newcommand{\Gbf}{\mathbf{G}}

\newcommand{\Hbb}{\mathbb{H}}

\DeclareMathOperator{\diam}{diam}
\DeclareMathOperator{\gr}{gr}

\DeclareMathOperator{\dist}{dist}

\DeclareMathOperator{\supp}{supp}

\newcommand{\N}{\mathbb{N}}
\newcommand{\R}{\mathbb{R}}
\newcommand{\C}{\mathbb{C}}

\newcommand{\spt}{\mathrm{spt}}

\newcommand{\Sing}{\mathrm{Sing}}

\newcommand{\toweakstar}{\overset{*}\rightharpoonup}

\newcommand{\todown}{\downarrow}
\newcommand{\eps}{\varepsilon}

\renewcommand{\eps}{\varepsilon}
\newcommand{\vphi}{\varphi}
\makeatletter
\renewcommand*\env@matrix[1][*\c@MaxMatrixCols c]{%
\hskip -\arraycolsep
\let\@ifnextchar\new@ifnextchar
\array{#1}}
\makeatother

\DeclareMathOperator{\Lip}{Lip}

\newcommand{\mres}{\mathbin{\vrule height 1.6ex depth 0pt width
	0.13ex\vrule height 0.13ex depth 0pt width 1.3ex}}

\def\vint_#1{\mathchoice%
{\mathop{\kern 0.2em\vrule width 0.6em height 0.69678ex depth -0.58065ex
		\kern -0.8em \intop}\nolimits_{\kern -0.4em#1}}%
{\mathop{\kern 0.1em\vrule width 0.5em height 0.69678ex depth -0.60387ex
		\kern -0.6em \intop}\nolimits_{#1}}%
{\mathop{\kern 0.1em\vrule width 0.5em height 0.69678ex depth -0.60387ex
		\kern -0.6em \intop}\nolimits_{#1}}%
{\mathop{\kern 0.1em\vrule width 0.5em height 0.69678ex depth -0.60387ex
		\kern -0.6em \intop}\nolimits_{#1}}}
\makeatletter
\newcommand*{\RangeX}{%
{%
	\mathpalette\@RangeOf{X}%
}%
}
\newcommand*{\@RangeOf}[2]{%
\sbox0{$\m@th#1\mathsf{#2}$}%
\mathsf{#2}%
\kern-\wd0 %
\mkern2.75mu\relax
\nonscript\mkern.25mu\relax
\mathsf{#2}%
}
\makeatother

\newcommand{\aveint}[2]{\mathchoice%
{\mathop{\kern 0.2em\vrule width 0.6em height 0.69678ex depth -0.58065ex
		\kern -0.8em \intop}\nolimits_{\kern -0.45em#1}^{#2}}%
{\mathop{\kern 0.1em\vrule width 0.5em height 0.69678ex depth -0.60387ex
		\kern -0.6em \intop}\nolimits_{#1}^{#2}}%
{\mathop{\kern 0.1em\vrule width 0.5em height 0.69678ex depth -0.60387ex
		\kern -0.6em \intop}\nolimits_{#1}^{#2}}%
{\mathop{\kern 0.1em\vrule width 0.5em height 0.69678ex depth -0.60387ex
		\kern -0.6em \intop}\nolimits_{#1}^{#2}}}

\newcommand{\restr}{\mres}
\newcommand{\pbf}{\mathbf{p}}

\newcommand{\defeq}{\vcentcolon = }
\newcommand{\Fsing}{\Ffrak_Q(T)}

\title[]{Flat interior singularities for area almost-\\minimizing currents}

\date{\today}

\author[M. Goering]{Max Goering}
\address{Max Planck Institute, Inselstraße 22, 04103 Leipzig, Germany}
\email{\href{mailto:max.goering@mis.mpg.de}{max.goering@mis.mpg.de}}

\author[A. Skorobogatova]{Anna Skorobogatova}
\address{Department of Mathematics, Fine Hall, Princeton University, Washington Road, Princeton, NJ 08540, USA}
\email{\href{mailto:as110@princeton.edu}{as110@princeton.edu}}

\makeindex

\begin{document}

\maketitle

\begin{abstract}
	    	    The interior regularity of area-minimizing integral currents and semi-calibrated currents has been studied extensively in recent decades, with sharp dimension estimates established on their interior singular sets in any dimension and codimension. In stark contrast, the best result in this direction for general almost-minimizing integral currents is due to Bombieri in the 1980’s, and demonstrates that the interior regular set is dense. 
          
          The main results of this article show the sharpness of Bombieri's result by constructing two families of examples of area almost-minimizing integral currents whose flat singular sets contain any closed, empty interior subset $K$ of an $m$-dimensional plane in $\R^{m+n}$. The first family of examples are codimension one currents induced by a superposition of $C^{k,\alpha_{*}}$ graphs with $K$ contained in the boundary of their zero set. The second family of examples are two dimensional area almost-minimizing integral currents in $\mathbb{R}^4$, whose set of branching singularities contains $K$.
\end{abstract}

\section{Introduction and main results}

Suppose that $T$ is an $m$-dimensional integral current in in $\R^{m+n}$. Recall that a point $x \in \spt(T)$ is called an \emph{interior regular point} if there is a ball $\Bbf_r (x)$ and there exists an $\alpha > 0$ so that  $\spt(T)$ is a $C^{1,\alpha}$ embedded submanifold of $\R^{m+n}$ without boundary in $\Bbf_r (x)$. A point $x \in \spt(T) \setminus \spt(\partial T)$ is called an \emph{interior singular point} if it is not regular. The set of interior singular points is denoted by $\Sing (T)$. Given an open subset $U \subset \R^{m+n}$, $T$ is called area-minimizing in $U$ if it satisfies
\[
    \|T\|(U) \leq \|T+\partial S\|(U),
\]
for any $(m+1)$-dimensional integral current $S$ supported in $U$, where $\|T\|$ denotes the $m$-dimensional mass measure induced by $T$.
Questions about the regularity of an $m$-dimensional area-minimizing integral current $T$ have fundamentally shaped the development of geometric measure theory and calculus of variations. In codimension $1$, i.e., when $n = 1$, the works of Allard, Almgren, Bombieri, De Giorgi, Giusti, Federer, Fleming, and Simons show that the Hausdorff dimension of $\Sing (T)$ is at most $m-7$, see \cite{maggi2012sets} for more detailed references.

In higher codimension, the situation is much more delicate. The monograph of Almgren \cite{Almgren_regularity} showed that the Hausdorff dimension of $\Sing (T)$ is at most $m-2$. This dimension bound is sharp in light of examples arising from holomorphic varieties with branching singularities, such as
\[
	\{w^Q = z^p : (z,w)\in \mathbb C^2\}\, ,
\]
where $p>Q\geq 2$ are coprime integers. Almgren's theory has since been simplified and made more transparent in the series of works~\cite{DLS_MAMS, DLS_multiple_valued, DLS14Lp, DLS16centermfld, DLS16blowup}.

In this article we study the regularity of area almost-minimizers. We say that an $m$-dimensional integral current $T$ is $(C_{0},2\alpha, R_{0})$-almost minimizing in an open subset $U\subset\R^{m+n}$ if it has the following property: for all $r \le R_{0}$ and $\Bbf_{r}(x_{0}) \Subset U$, 

\begin{equation}\label{e:am}
\|T\|(\Bbf_{r}(x_{0})) \le \|T + \partial S\|(\Bbf_{r}(x_{0})) + C_{0} r^{m+2 \alpha},
\end{equation}
for all $(m+1)$-dimensional integral currents $S$ with $\spt(S) \Subset \Bbf_{r}(x_{0})$. Extending \cite{Almgren-elliptic} from almost minimal varifolds to almost area-minimizing currents, Bombieri showed in \cite{Bombieri82} that the interior regular set is dense for any almost-minimizer with the following property in place of \eqref{e:am}: given a fixed compact subset $F$ of $U$, for all $r \le R_{0}$ and $\Bbf_{r}(x_{0}) \Subset F$ it holds that
\begin{equation}\label{e:am-Bombieri}
    \|T\|(\Bbf_{r}(x_{0})) \le \|T + \partial S\|(\Bbf_{r}(x_{0})) + C_{0} r^{2 \alpha}\|T + \partial S\|(F)
\end{equation}
for all $(m+1)$-dimensional integral currents $S$ with $\spt(S) \Subset \Bbf_{r}(x_{0})$. \footnote{In fact, Bombieri demonstrated this for a larger class of gauge functions $r\mapsto\omega(r)$ that satisfy a suitable Dini condition, which in particular includes $r\mapsto r^{2\alpha}$ for $\alpha>0$. This is however compensated by his definition of the regular set being that where $\spt(T)$ is locally merely $C^1$ in place of $C^{1,\alpha}$.}

This article shows that Bombieri's result cannot generally be strengthened, and demonstrates the impossibility of extending the known regularity theory for area minimizers to the class of $(C_0,2\alpha,R_0)$-almost minimizing currents. This is done by two {types of example} for which the collection of flat singularities of multiplicity $Q$, denoted by $\Ffrak_Q(T)$, is large. In general, it appears that the notion of $(C_{0},2\alpha, R_{0})$-almost minimality cannot be immediately compared to that introduced by Bombieri. However, our examples satisfy both notions of almost-minimality (see Remark \ref{r:Bom}). It is well-known that in the codimension $1$ and multiplicity $1$ setting (i.e., when $n=1=Q$), the singular set for area almost-minimizers satisfies the same dimension bounds as minimizers \cite{Tamanini84}; see also \cite{maggi2012sets} for a nice exposition on this vein of results. The main machinery in those settings is a ``flat implies smooth" $\eps$-regularity argument, which breaks down in higher codimension and in the presence of higher multiplicities. So, necessarily, both types of almost-minimizing integral currents constructed here will have a large {set of interior flat singularities}.

The first type of example, giving rise to Theorem \ref{t:graphs}, is a superposition of single-valued $C^{k,\alpha_*}$ graphs $\{f_{i}\}_{i=1}^{Q}$, for any $k\in \N$, $\alpha_{*} \in (0,1]$, inducing a {$(C_{0}, 2 \frac{k+\alpha_{*}-1}{k+\alpha_{*}}, R_{0})$-almost minimizing} $m$-dimensional integral current {T} in $\R^{m+1}$. {In this example, we prescribe a compact set with empty interior $K$ in the domains of $f_{i}$ so that $\Ffrak_Q(T)$ contains the set $K$.} Theorem \ref{t:graphs} is heavily motivated by the recent work of \cite{david2023branch} which demonstrates that almost-minimizers of the $2$-phase Bernoulli free boundary problem can have a larger set of branch points than that expected for minimizers. In that context, branch points are regular points of the free boundaries, but are points at which the two free boundaries have no better regularity than $C^{1,\alpha}$ and meet tangentially to form a ``cusp-like" structure.  

Bombieri's result \cite{Bombieri82} can be rephrased as saying that for any area almost-minimizing integral current $T$ (in the sense \eqref{e:am-Bombieri}), $\Sing(T)$ is a relatively closed subset of $\spt(T)$ with relatively empty interior. Theorem \ref{t:graphs} demonstrates that one may prescribe $\Sing(T)$ to contain an arbitrary such set embedded in $\R^{m} \times \{0\}^{n}$. 

\begin{theorem} \label{t:graphs}
Let $Q \in \N_{\ge 2}$, $\alpha \in (0,1)$, and $K \subset \R^{m} \times \{0\} \subset \R^{m+1}$ be a closed set with relatively empty interior. Choose $k \in \N$ and $\alpha_{*} \in (0,1]$ so that $\alpha = \frac{k + \alpha_{*}-1}{k+\alpha_{*}}$. Then there exists $C_{0}, R_{0} > 0$ depending only on $m$, $k$ and $\alpha_{*}$ and there exists an $m$-dimensional integral current $T = \Gbf_F$ where $F=\sum_{i=1}^{Q} \llbracket f_{i} \rrbracket$ with $f_{i} \in C^{k,\alpha_*}(\R^{m};\R)$ such that $T$ is $(C_{0}, 2 \alpha, R_{0})$-almost minimizing in $\R^{m+1}$ and has a singular set satisfying $\Sing(T) = \Fsing \supset K$.
\end{theorem}

Letting $\alpha=\frac{k+\alpha_*-1}{k+\alpha_*}$ and choosing $k$ arbitrarily large, Theorem \ref{t:graphs} produces a $(C_{0}, 2 \alpha, R_{0})$-almost minimizing integral current with a large singular set for $\alpha$ arbitrarily close to $1$. However, Theorem \ref{t:graphs} leaves open the question of whether or not there is a $(C_{0}, 2, R_{0})$-almost minimizing integral current with large singular set.

\begin{remark}\label{r:Q=1}
    Note that when $Q=1$, the integral current induced by \emph{any} $C^{1,\alpha}$ graph of an $\R^n$-valued function is $(C_0,2\alpha, \infty)$-almost minimizing in $\R^{m+n}$ for some $C_0,R_0>0$. However, no such current has any singularities, since by definition, there exists a neighborhood of every point where $\spt(T)$ is a $C^{1,\alpha}$ embedded submanifold of $\R^{m+n}$.
\end{remark}

In light of Remark \ref{r:Q=1}, Theorem \ref{t:graphs}, and \cite{Tamanini84}, one could reasonably wonder if any area almost-minimizing integral current is a superposition of sheets each with better regularity than the full current, at least up to a small singular set. In Theorem \ref{t:main-branch}, we show that this is not necessarily possible by constructing a family of area almost-minimizers which patch together re-scaled and translated cut-offs of the current induced by the branched holomorphic variety $\{z^{Q} = w^{Qk+1}\}$ in such a way that the branching singularities accumulate toward an arbitrary relatively closed subset with empty interior. 

\begin{theorem}\label{t:main-branch}
Let $Q \in \N_{\ge 2}, k \in \N$, and  $\alpha \defeq \frac{Qk+1-Q}{Qk+1}$. There exists $C_{0}, R_{0} > 0$ depending only on $k$ and $Q$, so that for any closed $K \subset \R^{2}\times\{0\}\subset\R^4$ with empty interior, there exists a $(C_{0}, 2 \alpha, R_{0})$-almost minimizing $2$-dimensional integral current $T$ in $\R^{4}$ with a (genuinely) branched singular set $\Sing(T) = \Fsing \supset K$.
\end{theorem}

In addition to showing that area almost-minimizing integral currents are not generally superpositions of regular surfaces like in Theorem \ref{t:graphs}, this shows that the genuine branch points of a two-dimensional area almost-minimizer can be prescribed to contain any closed subset of a two-dimensional plane with empty interior. In particular, not only is it in sharp contrast with the work of \cite{DLSS1,DLSS2,DLSS3}, where it is demonstrated that two-dimensional semi-calibrated currents have isolated singularities, but it also shows the drastic failure of the $(m-2)$-dimension bound \cite{Spolaor} for the set of genuine branching singularities for semi-calibrated currents. Semi-calibrated currents form a subclass of $(C_{0}, 1, R_{0})$-almost minimizing currents {for some $C_0,R_0>0$}, but the error from minimality has a very specific structure coming from the semi-calibration, therefore allowing for improved regularity. This article confirms that without the additional structure on the permitted error,\footnote{For example, having an elliptic PDE constraint on the current, like in the semi-calibrated case.} one cannot hope for an analogous regularity theory.

The following example demonstrates precisely how Theorem \ref{t:graphs} and Theorem \ref{t:main-branch} can be used to explicitly create an almost area-minimizing current with singular sets whose mass is arbitrarily close to that of the entire current.
\begin{example}
	Let $\{x_{i}\} \subset \pi_0\coloneqq \R^{m} \times \{0\}^{n} \subset \R^{m+n}$ enumerate the rational points, $Q =2$, and $\alpha \in (0,1)$. For each $\eps > 0$, write $B^{\eps}_{i} \coloneqq B_{\eps 2^{-i}}(x_{i})$. Then the set $K_{\eps} \defeq \pi_0 \setminus \cup_{i} B^{\eps}_{i}$ is relatively closed with empty interior in $\pi_0$, and for every $x \in \pi_0$ we have the lower estimate $|K_\eps \cap B_{r}(x)| \ge \omega_{m}\big(r^{m} -  \frac{\eps^{m}}{2^{m}-1}\big)$ for the $m$-dimensional Lebesgue measure of $K_\eps\cap B_r(x)$. 
	
	By Theorem \ref{t:graphs} (or Theorem \ref{t:main-branch} if $m=2=n$), there exists a current $T_{\eps}$ which is a $(C_{0}, 2 \alpha, R_{0})$ almost-minimizer for area whose singular set contains $K_{\eps}$ and with $C_{0}$ and $R_{0}$ independent of $\eps$. Moreover, $\|T_\eps\|(\Cbf_r(x,\pi_0)\cap K_\eps) = 2|K_\eps \cap B_{r}(x)|$ and $T_\eps$ is induced by the superposition of graphs of two Lipschitz functions whose Lipschitz constants converge to $0$ as $\eps$ converges to zero (see  Lemma \ref{t:uniformity}). Thus, the family of $(C_{0}, 2 \alpha, R_{0})$ almost-minimizers of area $\{T_{\eps}\}$ has the property that
	$$
	\lim_{r \uparrow \infty} \frac{\|T_{1/2}\| \left( \Cbf_{r}(x, \pi_0) \cap \Sing (T_{1/2}) \right)}{ \|T_{1/2}\| \left( \Cbf_{r}(x, \pi_0)  \right) }  = \lim_{\eps \downarrow 0} \frac{\|T_{\eps}\| \left( \Cbf_{1}(x, \pi_0) \cap \Sing (T_{\varepsilon}) \right)}{ \|T_{\eps}\| \left( \Cbf_{1}(x, \pi_0)\right)} = 1 \qquad \forall x \in \pi_0.
	$$
	That is, the singular set (branched singular set when $m=2=n$) has mass arbitrarily close to the entirety of $\|T_{\varepsilon}\|$ in $\Cbf_1(x,\pi_0)$ as $\varepsilon \downarrow 0$, and the same holds true for $T_{1/2}$ in $\Cbf_r(x,\pi_0)$ as $r\uparrow \infty$.
\end{example}

It is instructive to compare Theorem \ref{t:graphs} to the recent work \cite{Simon-fractal}, which shows that for any compact set $K\subset \R^{m-7}\times\{0\}\subset \R^{m+1}$, there exists a smooth metric on $\R^{m+1}$ and an $m$-dimensional stable minimal hypersurface whose singular set is $K$. A similar result in the framework of higher codimension area-minimizing {integral currents} was demonstrated by Liu in \cite{liu2021conjecture}, where it was shown that for any compact subset $K\subset \R^{m-2}$, there exists a smooth $(m+3)$-dimensional manifold $\Sigma$ and an $m$-dimensional homologically area-minimizing (calibrated) surface contained in $\Sigma$, whose interior singular set is $K$. However, in the former case, the singularities are modeled on Simons' cone, thus lying in the $(m-7)$-stratum, while in the latter case, the singularities are modeled on ``crossing-type" singularities formed by transverse self-intersections of the surface, thus lying in the $(m-2)$-stratum (cf. \cite[Remark 5]{liu2021conjecture}). In particular, these are not flat singularities. {Moreover, the ambient metric in both \cite{Simon-fractal} and \cite{liu2021conjecture} is merely smooth, not real analytic.} While stable minimal surfaces are not in general almost area-minimizing, stable minimal surfaces clearly enjoy significantly better regularity properties, in light of the contrast between \cite{W14} and this article. Nevertheless, the constructions in \cite{Simon-fractal, liu2021conjecture} demonstrate that even under more robust structural assumptions such as being area-minimizing or being a stable minimal hypersurface, fractal singular sets are permitted, albeit lower dimensional and with only a smooth ambient metric.

\subsection*{Acknowledgments}
Some of this work was completed while the authors visited SLMath (formerly MSRI), supported by the NSF grant FRG-1853993. A.S. acknowledges further support from FRG-1854147 and thanks Camillo De Lellis for his helpful comments and Reinaldo Resende for his interest and for numerous useful discussions.

\section{Background and preliminaries}
\subsection{Notation and background} \label{s:notation}

Throughout, we will consider $m$-dimensional integral currents in $\R^{n+m}$ which are induced by graphs of functions which are $C^{k,\alpha_{*}}$ for some $m,n,k \in \N$ and $\alpha_{*} \in (0,1]$. We let $C, C_0, C_1, \dots$  denote constants, whose dependencies  will be given when they are introduced. We will at times use the notation $\lesssim$ and $\simeq$, to suppress multiplicative constants. {If these constants
depend only on $m$ and $n$, there will be no subscript, meanwhile we include any other dependencies in subscripts.} {We let $\pi, \pi_{i}, \varpi$ denote $m$-dimensional planes {(namely, affine subspaces)} in $\R^{m+n}$, while and $\pi^{\perp}$ denotes the $n$-dimensional plane orthogonal to $\pi$. We denote by $\mathbf{p}_{\pi}: \mathbb R^{m+n}\to \pi$ the} orthogonal projection onto $\pi$, while $\mathbf{p}_{\pi}^\perp$ denotes the orthogonal projection onto $\pi^\perp$. $\Bbf_r(x)$ {and $\overline{\Bbf_r(x)}$ will respectively denote an open and closed} $(m+n)$-dimensional Euclidean ball of radius $r$ centered at $x$. {We let $B_r(x,\pi)$, $\overline{{B}_r(x,\pi)}$ respectively denote the open and closed $m$-dimensional disks $\Bbf_r(x)\cap \pi$ and $\overline{\Bbf_r (x)} \cap \pi$} of radius $r$ centered at $x$ in a given $m$-dimensional plane $\pi\subset\R^{m+n}$ passing through $x$; the dependency on the plane will be omitted if clear from context. We let $\Cbf_r(x,\pi)$ denote the $(m+n)$-dimensional cylinder $B_r(\mathbf{p}_\pi(x),\pi)\times \pi^\perp$. Given an $m$-dimensional current $T$, $\|T\|$ denotes the mass measure induced by $T$, while $\partial T$ denotes the $(m-1)$-dimensional current which is the boundary of $T$ and $\spt(T)$ denotes the support of the current. We use $\omega_m$ to denote the $m$-dimensional Hausdorff measure of the $m$-dimensional unit disk.

As an abuse of notation which should be easily discernible from context, $|\cdot|$ will denote the Euclidean norm of vectors, the norm induced by the metric on the Grassmannian $G(m,m+n)$ of $m$-dimensional linear subspaces of $\R^{m+n}$, and also the $m$-dimensional Lebesgue measure of subsets of a given $m$-dimensional plane. For a matrix $B$, we let $B^t$ denote its transpose.  $\Acal_Q(\pi^\perp)$ denotes the space of $Q$-tuples of vectors in $\pi^\perp$ (cf. \cite{DLS_MAMS}). We use the notation $\nabla$ for the gradient of a function on an $m$-dimensional plane, with respect to a canonical orthonormal choice of coordinates on that plane. Given Lipschitz functions $f:\Omega\to \pi_0^\perp$ and $F:\Omega\to \Acal_Q(\pi_0^\perp)$ on an open subset $\Omega$ of an $m$-dimensional plane $\pi_0\subset\R^{m+n}$, $\mathrm{gr}(f)\subset\R^{m+n}$ denotes the graph of $f$, while $\Gbf_F$ denotes the $m$-dimensional current induced by the multi-graph of $F$ (cf. \cite[Definition 1.10]{DLS_multiple_valued}). Given measurable functions $f_1,\dots,f_Q:\Omega\to \pi_0^\perp$, we write $F=\sum_{i=1}^Q \llbracket f_i \rrbracket$ for the $Q$-valued map with a decomposition into $f_i$ (cf. \cite[Definition 1.1]{DLS_MAMS}). {For a constant $\Lambda>0$, we let $\Lip_\Lambda$ denote the space of Lipschitz functions with Lipschitz constant bounded above by $\Lambda$.} We use $\supp(f)$ to denote the support of functions $f$.

\subsection{Key preliminary results}\label{s:prelim}

In this section we prove an important preliminary result (Proposition \ref{c:almostmin}), demonstrating that controlling pairwise gradient deviations for a superposition of $C^{1,\alpha}$ graphs guarantees a $(C_0, 2\alpha, R_0)$-almost minimality estimate in a ball of any sufficiently small radius in $\R^{m+n}$. This will be a key tool for proving both Theorem \ref{t:graphs} and Theorem \ref{t:main-branch}.

We first state the following important lemma, which not only will be crucial in the proof of  Proposition \ref{c:almostmin} below, but is in addition a powerful tool in its own right, due to the flexibility inherent in the domain $\Omega^{\prime}$ being any open set and the simplicity of verifying the hypotheses therein for explicitly constructed sets $\Omega^{\prime}$. Indeed, it will also be used independently in the proof of Theorem \ref{t:main-branch}.

\begin{lemma} \label{l:almostmin}
	Fix $q \in \N$ and $\alpha \in (0,1)$. Let $\pi\subset\R^{m+n}$ be an $m$-dimensional plane. Let $\Omega'$ be an open subset of $\pi$. Suppose that there exist $g_{i},\dots,g_q \in C^{1}( \Omega' ;\pi^\perp)$, points $\{x_{i}\}_{i=1}^{q} \subset \Omega^{\prime}$ and constants $C_{1},C_{2}, C_{3} > 0$ so that: 
	\begin{itemize}
		\item[(i)] $|\nabla g_{i_0}(x_{i_0})| \leq C_1 (\diam\Omega')^\alpha$ for some $i_0\in\{1,\dots,q\}$;
		\item[(ii)] for all $i\neq j$,
		\begin{equation}\label{e:C1alpha}
			|\nabla g_i(x_i) -\nabla g_j(x_i)| \le C_2 (\diam\Omega')^\alpha;
	    \end{equation}
        \item[(iii)] and  for all $y,z \in \Omega^{\prime}$, 
            $$
            | \nabla g_{i}(y) - \nabla g_{i}(z)| \le C_{3} (\diam \Omega^{\prime})^\alpha.
            $$
  \end{itemize}
  
	 Then there exists $\bar{C}=\bar{C}(C_1,C_{2}, C_{3})$ such that
    \begin{equation}\label{e:Dir}
    	\sum_{i=1}^{q}\int_{\Omega'} |\nabla g_i|^2 \leq \bar C q (\diam\Omega')^{2\alpha}|\Omega'|. 
    \end{equation}
\end{lemma}

\begin{proof}[Proof of Lemma \ref{l:almostmin}]
	Let $r\coloneqq \diam\Omega'$ and relabel indices if necessary so that property (i) holds for $i_0=1$. Then, for all $x \in \Omega'$ and all $1 \le i \le q$, we have
	\[
		|\nabla g_{i}(x)| \le |\nabla g_{i}(x) - \nabla g_{i}(x_{1})| + |\nabla g_{i}(x_{1}) - \nabla g_{1}(x_{1})| + |\nabla g_1(x_1)| \le (C_{3} + C_{2} + C_{1}) r^\alpha
	\]
	Squaring and integrating over $\Omega'$, this completes the proof.
\end{proof}

The following Proposition uses Lemma \ref{l:almostmin} to provide a sufficient condition for a $q$-valued graph to induce an area almost-minimizing integral current.

\begin{proposition} \label{c:almostmin}
    Fix $q \in \N$, $\alpha \in(0,1)$,  $r > 0$, and $x_0\in\R^{m+n}$. Denote $\R^m \equiv \R^m\times\{0\}^{n} \subset \R^{m+n}$ and $\R^n \equiv (\R^m\times \{0\}^n)^\perp$. Let $x'_0\coloneqq \pbf_{\R^m}(x_0)$ and $M > 0$. Then there exists a constant $R_{0}>0$ depending on $m,n$ and $M$ such that the following holds for each $r \leq R_0$. Suppose that $f_{1}, \dots, f_{q} \in {C^{1,\alpha}\cap \Lip_{1/4}(B_{2r}(x'_0);\R^n)}$ with $\max_i[ \nabla f_{i}]_{C^{0,\alpha}(B_{2r}(x_{0}^{\prime}))} \leq M$, {and $\gr(f_i)\cap\Bbf_r(x_0)\neq \emptyset$}. If for every pair of indices $i<j$ we have
    \begin{equation}\label{e:C1alpha.2}
		|\nabla f_i(x) -\nabla f_j(x)| \le C_4 |f_i(x)-f_j(x)|^\alpha \qquad \forall x\in B_{2r}(x_0'),
	\end{equation}
    then the multivalued function $F_1\coloneqq \sum_{i=1}^q \llbracket f_i \rrbracket$ satisfies
    \begin{equation}\label{e:alm-min-ball}
        \|\Gbf_{F_1}\|(\Bbf_r(x_0)) \leq \|\Gbf_{F_1}+\partial S\|(\Bbf_r(x_0)) + C_0 r^{m+2\alpha}
    \end{equation}
    for some constant $C_0 = C_0(C_4,m,n,q,\alpha,M)>0$, for all $(m+1)$-dimensional integral currents $S$ supported in $\Bbf_{r}(x_0)$.
\end{proposition}

\begin{remark} \label{r:smalllip}
If $f_{i} \in C^{1,\alpha_{*}}(\R^{m};\R^{n})$, the requirement $f_{i} \in \Lip_{1/4}(\R^{m};\R^n)$ could be removed. The reason this small Lipschitz requirement appears is due to formulating the proposition in a localized way so that information on the Hölder semi-norm at scale $2r$ implies a reparametrization at scale $r$. This is done for the sake of exposition in the proof of Theorem \ref{t:main-branch}. At the cost of making the ratio larger (and consequently $R_{0}$ smaller, see the hypothesis \eqref{e:closeenough} of Proposition \ref{p:step1}) the Lipschitz constant need only be finite, which is implied by $f_{i} \in C^{1,\alpha_{*}}(\R^{m} ; \R^{n})$.
\end{remark}

Although Proposition \ref{c:almostmin} merely requires $f_i \in C^{1,\alpha} ( B_{2r}(x_{0}^{\prime}))$, it is instructive to think of the functions $f_{i}$ in Proposition \ref{c:almostmin} as being in $C^{k,\alpha_{*}}$ for some $\alpha_*\in(0,1]$ and consider $\alpha = \frac{k+\alpha_* - 1}{k+\alpha_*}$. This is the context in which Proposition \ref{c:almostmin} will be applied when proving Theorem \ref{t:graphs} and Theorem \ref{t:main-branch}, and helps explain the presence of the exponent $\frac{k + \alpha_{*} - 1}{k+\alpha_{*}}$: if $f \in C^{k,\alpha_{*}}$ and $x'_{0} \in \R^{m} $ is such that $\partial^{\beta}f(x'_{0}) = 0$ for all multi-indices $\beta$ with $|\beta| \le k$, then

$$
\begin{cases}
    |f(x) - f(x'_{0})| \lesssim |x-x_{0}|^{k+\alpha_{*}} \\
    |\nabla f(x) - \nabla f(x'_{0})| \lesssim |x-x'_{0}|^{k + \alpha_{*} -1 } = \left( |x-x'_{0}|^{k+\alpha_{*}} \right)^{\frac{k + \alpha_{*} - 1}{k+\alpha_{*}}}.
\end{cases}
$$
See also Proposition \ref{p:dprop} for more insight on the relationship between $k,\alpha,\alpha_{*}$.

\begin{remark} \label{r:alphagap}
When $q=1$, \eqref{e:C1alpha.2} is a vacuous hypothesis. So, when $k=1$, the conclusion of Proposition \ref{c:almostmin} follows merely from the $C^{1,\alpha}$-regularity of the one function $f_i$. This recovers the sharp exponent between area almost-minimizing integral currents and $C^{1,\alpha}$ graphs, see \cite{duzaar2002optimal}.

When $q \geq 2$, the hypothesis \eqref{e:C1alpha.2} assumes roughly that despite being merely $C^{1,\alpha}$-regular, the graphs of $f_{j}$ meet tangentially in a way that their difference tangentially approaches zero like a $C^{1,\alpha_{*}}$ function (cf. Proposition \ref{p:dprop} below). When $\alpha < 1 = k$, we have $\alpha_{*} = \frac{\alpha}{2-\alpha} > \alpha = \frac{\alpha_{*}}{1+\alpha_{*}}\in (0,\frac{1}{2}]$. This means in the multi-sheeted setting, to conclude that the corresponding multi-valued graph is an area almost-minimizer with error exponent $2 \alpha$ our techniques require strictly more regularity on $f_j$ than in the single-sheeted setting. Nonetheless, this gap appears to be necessary in our methods, because it can happen that a ball $\Bbf_{r}(x_0)$ intersects two sheets $\llbracket \gr(f_{1}) \rrbracket, \llbracket \gr(f_2) \rrbracket$ that are mutually disjoint in $\Bbf_r(x_0)$, with $|\nabla f_{1}(x) - \nabla f_{2}(x)| \gg r^{\alpha}$ for all $x \in B_r(\pbf_{\pi_0}(x_0))$.
\end{remark}

The following proposition demonstrates that in codimension one, the hypothesis \eqref{e:C1alpha.2} in Proposition \ref{c:almostmin} follows from the pointwise property of pairwise tangential touching for a collection of $C^{1,\alpha_*}$ graphs. We will not make use of this proposition in the proofs of the main results. However, not only is it instructive in gaining a deeper understanding behind the relationship of the exponents $\alpha$ and $\alpha_*$, but it also allows one to conclude that \emph{any} superposition of ordered $C^{1,\alpha_*}$-graphs is area almost-minimizing; see Remark \ref{r:ordered-am} below.

\begin{proposition}\label{p:dprop}
	Fix $Q >1$, and $\alpha_*\in(0,1]$. Let $\R^m \equiv \R^m\times\{0\}\subset \R^{m+1}$ and let $\R \equiv (\R^m\times \{0\})^\perp$. Suppose that $f_{1},\dots, f_Q \in C^{1,\alpha_*}\cap \Lip_{1/4}( \R^m ;\R)$ and that $f_{i}$ satisfies the property for all $i 
 \neq j$:
	\begin{equation}\label{e:dprop}
		f_{i}(x) = f_{j}(x) \implies \nabla f_{i}(x) = \nabla f_{j}(x) \qquad \forall x\in \R^m.
	\end{equation}
	Choose $C_4 = 2 \max \{1, \max_{i\neq j}[\nabla f_i-\nabla f_j]_{C^{0,\alpha_*}} \}$. Then the functions $\{f_{i}\}$ satisfy \eqref{e:C1alpha.2} with $\alpha = \frac{\alpha_*}{1+\alpha_*}$ for all $x \in \R^m$, with this choice of $C_4$. In particular, for $F \coloneqq \sum_{i=1}^Q \llbracket f_{i} \rrbracket$, the current $\Gbf_{F}$ is a $(C_0, 2\alpha, R_0)$-almost minimizing current in $\R^{m+1}$, for $C_0$, $R_0$ given by Proposition \ref{c:almostmin}.
\end{proposition}

We again recall that the assumption $f_i\in\Lip_{1/4}(\R^{m};\R)$ could be dropped at the expense of shrinking $R_{0}$, see Remark \ref{r:smalllip}.

\begin{remark}\label{r:ordered-am}
    If $f_{1} \le \dots \le f_{Q}$ are $C^{1,\alpha_{*}}(\R^{m} ; \R)$ functions, then they necessarily satisfy \eqref{e:dprop}. In particular, Proposition \ref{p:dprop} and Remark \ref{r:smalllip} imply that the current $T=\Gbf_F$ associated to the graph of $F = \sum_{i=1}^{Q} \llbracket f_{i} \rrbracket$ is a $(C_{0}, 2 \alpha, R_{0})$-almost minimizing current for some $C_0,R_0>0$.
\end{remark}

\begin{proof}[Proof of Proposition \ref{c:almostmin}]
 Fix a ball $\Bbf_r(x_0)$ as in the statement of the proposition. Let $\pi_{1}$ denote the $m$-dimensional tangent plane to $\Gbf_{f_1}$ at some point $(y_0,f_1(y_0))\in\Bbf_{r}(x_0)$.

    We will apply Proposition \ref{p:step1} to the plane $\varpi = \pi_{1}$ to reparametrize each $f_{i}$ over $B_{r}(\pbf_{\pi_{1}}(x_{0});\pi_{1})$ by some functions $g_{i}$, which we will show satisfying the hypotheses of Lemma \ref{l:almostmin}. Since these $g_{i}$ will have the same graph as $f_{i}$ after intersecting with $\Bbf_{r}(x_0)$, the conclusion of Lemma \ref{l:almostmin} confirms \eqref{e:alm-min-ball} and will complete our proof.

    We first check the hypotheses of Proposition \ref{p:step1}, under the assumptions 
    $$\delta= \frac{1}{2}, \quad \sigma = \frac{1}{2}\left[1 + \max_i\Lip(f_i)\right], \quad \max_i \Lip(f_i) \le \frac{1}{4}.$$ 
    Let $A : \R^{m} \to \R^{n}$ have graph parallel to $\pi_{1}$. Since $\nabla A = \nabla f_{1}(y_0)$, it follows that
    $\Lambda \coloneqq \max_i\Lip(f_{i} - A) \le 2\max_i\Lip(f_i) \leq \frac{1}{2}$. Moreover, since $\gr(f_{i}) \cap \Bbf_{r}(x_0) \neq \emptyset$ there exists $x_i\in B_r(x'_0)$ with $|(x_{i},f_{i}(x_{i})) - x_{0}| \le r$. Writing $x_{0} = (x'_0, \bar{x})$, we have $|f_{i}(x'_{0}) - \bar{x}| \le |x'_0-x_{i}| \Lip(f_{i}) + |f_{i}(x_i) - \bar{x}| \le (1+ \frac{1}{4}) r$.

    So, the choice of $\delta = \frac{1}{2}$ is indeed valid since the above assumptions yield
    $$
    \frac{1}{2} \le  \frac{1}{\sqrt{1 + \Lambda^{2}}} - \frac{\sigma \max_{i} \Lip(f_i)}{\sqrt{1+\max_{i} \Lip(f_i)^{2}}},
    $$
    confirming \eqref{e:closeenough}. Since by assumption $f_{i} \in C^{1,\alpha}(B_{2r}(x'))$ we can apply Proposition \ref{p:step1} to construct $g_{i} \in C^{1,\alpha}( B_{r}(\pbf_{\pi_{1}}(x_0));\pi_{1}); \pi_{1}^{\perp})$ so that $\gr(g_{i}) = \gr(f) \cap \Cbf_{r}(x_0;\pi_{1})$. In particular, if $G = \sum_{i=1}^{q}, \llbracket g_{i} \rrbracket$ then 
    \begin{equation} \label{e:preconstancy}
        {\Gbf_G  = \Gbf_{F_1} }\mres \Cbf_{r}(x_{0};\pi_1). 
    \end{equation}
    Thus, $	\partial \left( \left(\pbf_{\pi_{1}}\right)_{\sharp} \Gbf_{F_1} \right)\restr B_{r}(\pbf_{\pi_{1}}(x_{0})) = 0$. Hence, by \eqref{e:preconstancy} and the fact that $G$ is $q$-valued it follows that
    \begin{equation} \label{e:0thorder}
    (\mathbf{p}_{\pi_1})_\sharp\Gbf_{F_1}\mres \Cbf_r(x_0,\pi_1) = q\llbracket B_r(\mathbf{p}_{\pi_1}(x_0))\rrbracket.
    \end{equation}
    We now claim that for $r$ small enough the $g_{i}$ satisfy the hypotheses of Lemma \ref{l:almostmin}. Indeed, by our choice of $\pi_{1}$, {we have $\nabla f_{1}(y_0) = \nabla A(y_0)$, which in particular tells us that $\nabla g_1(\zeta_0)=0$ at some point $\zeta_0\in B_r(\pbf_{\pi_{1}}(x_0))$. Proposition \ref{p:step1}(2), (3), and \eqref{e:hseminorm} and the hypotheses $\max_i\Lip(f_i)\leq \frac{1}{4}$ and $\max_i [\nabla f_i]_{C^{\alpha}}\leq M$ tells us that for each $x\in B_{r}(\pbf_{\pi_{1}}(x_0))$  we have
    \begin{equation}\label{e:lemma2.1(i)}
    	|\nabla g_1(x)| \lesssim_\alpha [\nabla g_1]_{C^{\alpha}} r^\alpha \leq C [\nabla f_1]_{C^{\alpha}} r^\alpha \leq C_1 r^\alpha,
    \end{equation}
    where $C_1= C_1(\Lip(f_1),m,n,\alpha,M)$. This confirms hypothesis (i) of Lemma \ref{l:almostmin}.
    
{Now, for any $x\in B_{2r}(x'_0)$ we have
    \begin{align*}
    	|\nabla f_i(x) - \nabla f_j(x)| &\lesssim_{ C_4}|f_i(x) - f_j(x)|^\alpha \\
     & \lesssim_{\alpha} |f_i(x) - f_i(x_i)|^\alpha + |f_i(x_i)-f_j(x_j)|^\alpha + |f_j(x_j)-f_j(x)|^\alpha \\
    	&\le (1+\max_i \Lip(f_i)) r^\alpha \lesssim r^\alpha
	\end{align*}
    { which shows that, due to the choice of $\delta=\frac{1}{2}$, $\Lambda\leq \frac{1}{2}$, and $M$, the right-hand side of \eqref{e:relgrad} can be controlled by 
    $$|f_{i}(y_{i}^{-1}) - f_{j}(y_{j}^{-1})|^{\alpha} + r^{\alpha},$$
    with constant depending on $M, m,n,\alpha$. Since each $f_{i}$ is Lipschitz and has graph intersecting $\Bbf_{r}(x_0)$, the difference in values between $f_{i}$ and $f_{j}$ can also be controlled by a constant multiple of $r$. Therefore \eqref{e:relgrad} guarantees

    }
    \begin{equation}\label{e:diff-g}
        |\nabla g_i(z) - \nabla g_j(z)| \lesssim_{M,m,n} r^\alpha,
    \end{equation}
    thus verifying the hypothesis (ii) of Lemma \ref{l:almostmin}. Meanwhile, hypothesis (iii) of Lemma \ref{l:almostmin} follows from \eqref{e:lemma2.1(i)}, \eqref{e:diff-g} with $j=1$ and the triangle inequality.
 }

}
    
 Thus we may apply Lemma \ref{l:almostmin} with $\Omega' = B_r(\mathbf{p}_{\pi_{1}}(x_0))$ to deduce that
	\begin{equation} \label{e:dir}
		\sum_{i=1}^{q}\int_{B_{r}(\pbf_{\pi_{1}}(x_0))} |\nabla g_i|^2 \leq \bar C q r^{2\alpha}|B_{r}(\pbf_{\pi_{1}}(x_0))|.
	\end{equation}

Since $\partial S$ has empty boundary, invoking \eqref{e:0thorder} in turn yields
\begin{align*}
		\|\Gbf_{F_1}\|(\Cbf_{r}(x_0,\pi_1))  
		& = \sum_{i=1}^{q} \|\Gbf_{g_i}\|(\Cbf_r(x_0,\pi_1)) \\
		&\leq q|B_r(\mathbf{p}_{\pi_1}(x_0))| + \frac{1}{2}\sum_{i=1}^{q}\int_{B_{r}(\pbf_{\pi_{1}}(x_0))} |\nabla g_i|^2 \\
        &\qquad+ \int_{B_{r}(\pbf_{\pi_{1}}(x_1))} O\big(|\nabla g_i|^4\big) \\
		&\leq q(1+Cr^{2\alpha})|B_r(\mathbf{p}_{\pi_1}(x_0))| \\
        & \le \|\Gbf_{F_{1}} + \partial S\|(\Cbf_{r}(x_0,\pi_1)) +C_0  r^{2\alpha}|B_r(\mathbf{p}_{\pi_1}(x_0))|.
	\end{align*}

Subtracting $\| \Gbf_{F_{1}}\|\left(\Cbf_{r}(x_{0}, \pi_{1}) \setminus \Bbf_{r}(x_{0})\right)$ from each side and recalling that $\Gbf_{F} \restr \Bbf_{r}(x_{0}) = \Gbf_{F_{1}} \restr \Bbf_{r}(x_{0})$ 
yields
\begin{equation}\label{e:am-concl}
\|\Gbf_{F}\|(\Bbf_{r}(x_{0})) \le  \| \Gbf_{F} + \partial S\|(\Bbf_{r}(x_{0})) + C_{0} r^{m+2 \alpha},
\end{equation}
  {for some $C_0=C_0(m,n,M,\alpha)>0$}.  
\end{proof}

\begin{proof}[Proof of Proposition \ref{p:dprop}]
	Fix two indices $i\neq j$ and let $g\coloneqq f_i - f_j$. We will show that for any $x\in\R^m$ it holds that
	\begin{equation}\label{e:C1alpha-scalar}
		|\nabla g(x)| \le 2\max\left\{1, [\nabla g]_{C^{\alpha_*}(B_\eps(x))}\right\} g(x)^{\alpha} \equiv 2\max\left\{1, [\nabla g]_{C^{\alpha_*}(B_\eps(x))}\right\} g(x)^{\frac{\alpha_*}{1+\alpha_*}}.
	\end{equation}

If $g$ is constant, there is nothing to prove. Otherwise, suppose to the contrary that, after translating, $g(0) =R^{1+\alpha_*}> 0$, but $|\nabla g(0)| > 2 C_{g} R^{\alpha_*}$ where $C_{g} \defeq \max\{1,[\nabla g]_{C^{\alpha}(B_\eps)}\}$. Further, by rotation, we may assume that $|\nabla g(0)| = \partial_1 g(0) > 2 C_g R^{\alpha_*}$. Then for every $\xi \in [-R e_1,0]\cap\supp (g)$, the $C^{1,\alpha_*}$-regularity of $g$ yields $|\partial_1 g (0) - \partial_1 g(\xi)| \le  C_{g} R^{\alpha_*}$, which in turn implies
	\begin{equation} \label{e:fprimeeps}
		  \partial_1 g(\xi) > C_{g} R^{\alpha_*} \qquad \forall \xi \in [-Re_1,0].
	\end{equation}
	If $[-R e_1,0]\subset \supp(g)$, the Fundamental Theorem of Calculus yields
	$$
	R^{1+\alpha_*} - g(- R e_1) = \int_{-R}^{0} \partial_{1} g( t e_{1})  dt \overset{\eqref{e:fprimeeps}}{>} C_{g} R^{1+\alpha_*}.
	$$
	Hence $g(-R e_1) < (1- C_{g}) R^{1+\alpha_*} \le 0$. By the mean value theorem there exists $t \in [- R, 0]$ so that $g(t e_1) = 0$. By hypothesis, it follows $\partial_1 g(t e_1) = 0$ contradicting \eqref{e:fprimeeps} in this case. 

    On the other hand, if $[-R e_1,0]$ is not contained in $\supp(g)$, then we can immediately find $t\in [-R,0]$ with $g(te_1)=\partial_1 g(t e_1) = 0$, again contradicting \eqref{e:fprimeeps}.

    Thus, we conclude that \eqref{e:C1alpha-scalar} indeed holds for every $x\in\R^m$. Since by assumption, $f_{i} \in C^{1,\alpha}\cap\Lip_{1/4}(\R^{m};\R)$, the hypotheses of Proposition \ref{c:almostmin} are satisfied. Applying Proposition \ref{c:almostmin} completes the proof.
\end{proof}

\section{Proof of main Theorems}\label{s:pf}

We will frequently be using the notion of a Whitney decomposition $\Wscr$ of $\R^m\setminus E$, for a given closed set $E\subset \R^m$ {(or more generally, for a closed subset of an $m$-dimensional plane $\pi_0\subset \R^{m+n}$)}. Such a decomposition $\Wscr$ consists of a family of closed dyadic cubes $L \subset \R^m \setminus E$ with
\begin{itemize}
	\item $\dist(L, E) \simeq \ell(L)$;
	\item each cube $L$ intersects at most $\Lambda=\Lambda(m)\in \N$ cubes in $\Wscr$;
	\item there exists $\lambda=\lambda(m)>0$ such that if $L_1,L_2\in \Wscr$ satisfy $L_1\cap L_2\neq \emptyset$ then $\lambda^{-1} \ell(L_2) \leq \ell(L_1)\leq \lambda \ell(L_2)$.
\end{itemize}

Before starting the proof, given $k \in \N$ and $\alpha_{*} \in (0,1]$, we construct {a regularized $(k+\alpha_{*})$-power of the distance} to a closed set $E \subset \R^{m}$, denoted by $\eta = \eta_{k,\alpha_{*},E}$. This function, defined below (see \eqref{e:smoothd}), is in direct analogue with the regularized distance function $\Delta$ in \cite[VI.2, Theorem 2]{SteinSI}. Crucially, in Lemma \ref{t:uniformity} we show that for any closed set $E$,  the function $\eta$ enjoys the following properties
    \begin{equation} \label{e:i}
        \eta(x) \simeq_{m,k,\alpha_*} \dist(x, E)^{k+\alpha_*}, \qquad x\in \R^m,
    \end{equation}
    and for any multi-index $\beta$,
    \begin{equation} \label{e:ii} |\partial^\beta \eta(x)| \lesssim_{|\beta|,m,k,\alpha_*} \dist(x,E)^{k+\alpha_*-|\beta|}, \qquad x\in \R^m
    \end{equation}

    Let $\sW$ be a Whitney decomposition of $\R^{m} \setminus E$ and let $\varphi \in C_{c}^{\infty}(\R^{m} ; [0, \infty))$ be a smooth cut-off function satisfying $\varphi \equiv 1$ on $[-0.5,0.5]^m$, $\supp(\varphi) \subset [-0.6,0.6]^m$. Consider
\begin{equation} \label{e:smoothd}
\eta_{k,\alpha_*,E}(x) \defeq \sum_{L \in \sW} \ell(L)^{k+\alpha_*} ~ \varphi \left( \frac{x-x_{L}}{\ell(L)} \right),
\end{equation}
where $x_{L}$ denotes the center of $L$ and $\ell(L)$ is the side-length of $L$. Note that the properties of a Whitney decomposition guarantee that the sum in \eqref{e:smoothd} is locally finite, with a uniformly bounded number of summands.

\begin{lemma} \label{t:uniformity}
Let $k \in \N$, $M>0$, $\alpha_{*} \in (0,1]$, and $\alpha = \frac{k+\alpha_{*}-1}{k+\alpha_{*}}$. Let $E \subset \R^{m}$ be a closed set. If $\eta\equiv \eta_{k,\alpha_*,E}$ is as in \eqref{e:smoothd}, then $\eta\in C^\infty(\R^m\setminus E)$ satisfies \eqref{e:i} and \eqref{e:ii}.
If $E$ additionally satisfies $\dist(x,E) \le M$ for all $x \in \R^{m}$, then
$$
\|\eta \|_{C^{k,\alpha_{*}}(\R^{m})} \lesssim_{m,k,\alpha_{*},M} 1.
$$
\end{lemma}

\begin{proof}
    Let us first demonstrate the validity of \eqref{e:i} and \eqref{e:ii}. Note that \eqref{e:i} follows trivially in the case where $x\in E$. Now fix $x\in \R^m\setminus E$. Observe that if $x\in 1.2 L$ for $L\in \Wscr$, then
    \[
        \ell(L) \simeq \dist(L,E) \lesssim \dist(x,E) \lesssim \dist(L,E) + \ell(L) \lesssim \ell(L),
    \]
    and if $x \not \in 1.2 L$, $\varphi \left( \frac{x-x_{L}}{\ell(L)} \right) = 0$.
    Combining this with the boundedness of $\vphi$ and the local finiteness of the sum in \eqref{e:smoothd} yields the conclusion of \eqref{e:i}. The conclusion of \eqref{e:ii} follows by entirely analogous reasoning, combined with the boundedness of $\partial^{\beta} \varphi$ and the fact that
    \begin{equation} \label{e:gradd}
        \partial^\beta \eta(x) = \sum_{L \in \Wscr} \ell(L)^{k+\alpha_*-|\beta|}\partial^\beta\vphi\left(\frac{x-x_L}{\ell(L)}\right),
    \end{equation}
    for every multi-index $\beta$.
    
        In particular, whenever $|\beta| \le k$, \eqref{e:ii} implies that
        \begin{equation} \label{e:ine} \lim_{y \to x} |\partial^{\beta} \eta(y)| = 0 \qquad \forall x \in E,
        \end{equation}
        verifying that $\| \eta \|_{C^{k}(\R^{m})} \lesssim_{m,k,\alpha_{*},M} 1$. So, it only remains to check that for any $|\beta| = k$, $[\partial^{\beta} \eta]_{C^{\alpha_{*}}(\R^{m})} \lesssim_{m,k,\alpha_{*},M} 1$.

    Note that \eqref{e:ine} confirms there is nothing to show when $x,y \in E$. So, for the remainder of the proof it suffices to consider, without loss of generality, that $x \in \R^{m} \setminus E$.  We turn our attention to the case when $\max \{\dist(x,E), \dist(y,E) \} \le |x-y|$. Since $\partial^{\beta} \varphi$ is bounded, \eqref{e:gradd} yields

\begin{equation}\label{e:crude}
|\partial^{\beta} \eta(x) - \partial^{\beta} \eta(y)| \lesssim \sum_{L \in\Wscr} \max\{\dist(x,E), \dist(y,E)\}^{\alpha_*} \lesssim |x-y|^{\alpha_*}.
\end{equation}
When $\min\{\dist(x,E),\dist(y,E)\} \ge |x-y|$, we first note that the definition of a Whitney decomposition guarantees that all the sums are taken over only those cubes $L \in \Wscr_x \cup \Wscr_y$ where $\Wscr_x = \{ L \in \sW : x \in 1.2 L \}$ and $\Wscr_y$ is defined analogously. In particular, all sums are finite and the number of summands is bounded uniformly by a dimensional constant independent of $x,y$. Exploiting this and the Lipschitz regularity of $\partial^\beta\vphi$, \eqref{e:gradd} implies that
\begin{align}
\nonumber    | \partial^\beta \eta(x) - \partial^\beta \eta(y)|  &\lesssim \sum_{L \in \Wscr_x} \ell(L)^{\alpha_*} \left|\frac{x-y}{\ell(L)}\right| + \sum_{L \in \Wscr_y} \ell(L)^{\alpha_*} \left|\frac{x-y}{\ell(L)}\right| \\
\nonumber    &\lesssim \sum_{L \in \Wscr_x} \dist(x,E)^{\alpha_*} \left|\frac{x-y}{\dist(x,E)}\right| + \sum_{L \in \Wscr_y} \dist(y,E)^{\alpha_*} \left|\frac{x-y}{\dist(y,E)}\right| \\
\nonumber    &= |x-y|^{\alpha_*}\left[\sum_{L\in \Wscr_x} \left|\frac{x-y}{\dist(x,E)}\right|^{1-\alpha_*} + \sum_{L \in \Wscr_y}\left|\frac{x-y}{\dist(y,E)}\right|^{1-\alpha_*}\right] \\
\label{e:care}    &\lesssim |x-y|^{\alpha_*}.
\end{align}

It remains to deal with the case where, without loss of generality, $$\dist(y,E) \leq |x-y| \leq \dist(x,E).$$ 
To this end, we argue as in \eqref{e:care} for cubes $L \in \Wscr_{x}$ and as in \eqref{e:crude} for cubes $L \in \Wscr_{y}$:
\begin{align*}
    | \partial^\beta \eta(x) - \partial^\beta \eta(y)|  &\lesssim \sum_{L\in \Wscr_x} \ell(L)^{\alpha_*} \left|\frac{x-y}{\ell(L)}\right| + \sum_{L\in \Wscr_y} \ell(L)^{\alpha_*}  \\
    &\lesssim |x-y|^{\alpha_*}.
\end{align*}
\end{proof}

\begin{remark}[Validity of Bombieri's almost-minimality definition]\label{r:Bom}
	Observe that in place of the estimate \eqref{e:am-concl} in Proposition \ref{c:almostmin}, one can use the preceding calculations and \eqref{e:0thorder} to instead establish the estimate 
	\begin{equation}
		\|\Gbf_{F}\|(\Bbf_{r}(x_{0})) \le  \| \Gbf_{F} + \partial S\|(\Bbf_{r}(x_{0})) + C r^{2 \alpha}\| \Gbf_{F} + \partial S\|(\Cbf_{r}(x_{0},\pi_1)).
	\end{equation}
	Combining with the property that $\dist(x,E)\leq 1$ for the set $E$ chosen in the proof of Theorem \ref{t:graphs} below, if the set $K$ is taken to be compact (rather than merely closed), the choice of $\Gbf_F$ therein satisfies the property \eqref{e:am-Bombieri} with the compact set $F$ taken to be $\overline{B_{M}(0,\pi_0)}\times [0,C_{m,k,\alpha_{*}}]\subset \R^{m+1}$, where $C_{m,k,\alpha_{*}}$ is a positive constant depending on $m,k,\alpha_*$ coming from \eqref{e:i},  with a choice of $M$ sufficiently large so that $K \Subset \overline{B_{M}(0,\pi_0)}$, for example.
	
	Likewise, we can establish the almost-minimality property \eqref{e:am-Bombieri} for the current $T$ in Theorem \ref{t:main-branch}.
\end{remark}

\begin{proof}[Proof of Theorem \ref{t:graphs}]

Let $K \subset \R^{m} \times \{0\}$ be as in Theorem \ref{t:graphs}, let $E = K \cup \{ x \in \R^{m} : \dist(x,K) \ge 1 \}$. Then $E$ satisfies the hypotheses of Lemma \ref{t:uniformity} with $M=1$. Let $\eta\equiv \eta_{k,\alpha_*,E}$ be given by Lemma \ref{t:uniformity} for this choice of $E$. Then, for $1 \le i \le Q$, 
consider the functions
\begin{equation} \label{e:fi2}
f_{i}(x) = \frac{i \eta(x)}{4 Q \Lip(\eta)}.
\end{equation}
Then, $\max_{i}\Lip(f_{i}) \le \frac{1}{4}$. In addition, since $\alpha < \alpha_{*}$, Lemma \ref{t:uniformity} tells us that $\max_i [\nabla f_i]_{C^{0,\alpha}}$ $\lesssim_{m,k,\alpha_*} 1$. So, Proposition \ref{c:almostmin} can be applied with some $M$ depending on $m,k,\alpha_*$.
Moreover, a direct computation shows that for all $x \in \R^m$,
\begin{align*}
    |\nabla f_{i}(x) - \nabla f_{j}(x)| &\overset{\eqref{e:ii}}{\lesssim}_{k,m,\alpha_{*},Q} |i-j|^{\frac{k+\alpha_{*}-1}{k+\alpha_{*}}}\dist(x,E)^{k+\alpha_{*}-1} \\
    &\overset{\eqref{e:i}}{\simeq_{k,m,\alpha_*},Q}  |f_{i}(x) - f_{j}(x)|^{\frac{k+\alpha_{*}-1}{k+\alpha_{*}}}.
\end{align*}
This demonstrates the hypotheses of Proposition \ref{c:almostmin} are satisfied with some constant $C_{4}$ depending only on $m,k,\alpha_{*}$, and $Q$. From Proposition \ref{c:almostmin} it follows that there exists some $R_{0}, C_{0}$ depending only on $m, k,\alpha_{*},$ and $Q$ so that \eqref{e:alm-min-ball} holds for any ball $\Bbf_{r}(x_0) \subset \R^{m+1}$ whenever $r \le R_{0}$. 
Consequently, the current $T = \Gbf_{F}$ with $F = \sum_{i=1}^{Q} \llbracket f_{i} \rrbracket$ is $(C_{0}, 2 \alpha, R_{0})$-almost minimizing with constants depending only on $m,k,\alpha_{*}$, and $Q$.

To see that $K\subset\Sing(T)$, simply notice that $K\subset\partial E$, and for any $x \in \partial E$, $\spt(T)$ is the union of $Q$ distinct $C^{1,\alpha_*}$-manifolds intersecting at $x$, and thus is $\spt(T)$ is not an embedded $C^{1,\alpha_*}$-manifold in any neighborhood of $x$. That is, $\partial E = \Sing (T)$. Finally, by \eqref{e:fi2} it follows that $\Sing(T) = \Fsing$; this completes the proof.
 \end{proof}

\begin{proof}[Proof of Theorem \ref{t:main-branch}]
	
Let $\alpha=\frac{Qk+1-Q}{Qk+1}$, $\pi_0\coloneqq \R^2\times \{0\}^{2} \subset\R^4$,  $z=(x,y)$ denote coordinates in $\pi_0\cong \R^2$, and $K\subset \pi_0$ be an arbitrary closed set with empty interior. We proceed to construct a multi-valued function $F: B_1(0,\pi_0) \to \Acal_2(\pi_0^\perp)$ whose reparametrization to an appropriate plane will satisfy the assumptions of Lemma \ref{l:almostmin}.  

Let $\Wscr$ be a Whitney decomposition of $\pi_0\setminus K$  in $\pi_0$, $\{L_l\}_{l\in \N}$ be an enumeration of the cubes in $\Wscr$ with $\ell(L_l) \le 1$, $z_l$ denote the center of $L_l$, and $r_l \coloneqq \frac{\ell(L_l)}{4} \le \frac{1}{4}$.
Define $v: \pi_0 \to \Acal_Q(\pi_0^\perp)$ as
\[
v(z) = \sum_{\xi^Q = z^{Qk+1}}\llbracket \xi \rrbracket \eqqcolon \sum_{j=1}^Q \llbracket v_j(z)\rrbracket,
\]
the $Q$-valued function whose graph is the holomorphic variety $\{w^Q=z^{Qk+1}\}\subset \C^2 \equiv \R^4$. Here, $v_j:\pi_0\to\pi_0^\perp$ are measurable functions representing the $Q$ roots of $z\mapsto z^{Qk+1}$.

 Let $\eta\equiv \eta_{k,\frac{1}{Q},E}:\pi_0\to \R$ be the function given by Lemma \ref{t:uniformity} for the closed set $E\coloneqq \pi_0\setminus B_{1/2}(0)$. Note that Lemma \ref{t:uniformity} implies that $\eta\in C_c^{k,\frac{1}{Q}}(\pi_0;\R)$ with $\|\eta\|_{C^{k,\frac{1}{Q}}}\lesssim_{k,Q} 1$. We note for later use, that $\supp(\eta)=  \overline{B_{1/2}(0)}$ and from \eqref{e:i} and \eqref{e:ii} it follows that
$|\nabla \eta(z)|  \lesssim_{Q,k} \eta(z)^{\alpha}.$ 
In particular,
\begin{equation} \label{e:etabounded}
	\left| \nabla \eta(z) \right| \eta(z)^{-\alpha} \lesssim_{Q,k} 1.
\end{equation}

Consider $w\coloneqq \eta v : B_1(0) \to \Acal_Q(\pi_0^\perp)$. Define the measurable functions $w_l : B_{r_{l}}(z_{l}) \to \mathcal{A}_{Q}(\pi_{0}^{\perp})$ by
\[
w_l(z) \coloneqq \kappa r_l^{\frac{Qk+1}{Q}}  w\left(\frac{z - z_l}{r_l}\right) = \sum_{j=1}^Q \left\llbracket \kappa r_l^{\frac{Qk+1}{Q}} \eta\left(\frac{z - z_l}{r_l}\right) v_j\left(\frac{z - z_l}{r_l}\right)\right\rrbracket,
\]
where $\kappa$ depends only on $Q,k$ and is chosen so that given any choice of branch cut for the logarithm on $B_{r_{l}}(z_l)$, each branch of $w_{l}$ satisfies $|\nabla w_{l}| \le \frac{1}{4}$, see \eqref{e:theconstant}. Now let $\Ucal = \bigcup_{l \in \N} B_{r_l}(z_l)$ and define the measurable $Q$-valued function $F:\pi_0 \to \Acal_Q(\pi_0^\perp)$ via a selection of measurable functions $g_j:\pi_0\to \pi_0^\perp$ as follows:
	\begin{equation} \label{e:keytoentiretyofcasec}
		F(z) = \sum_{j=1}^Q \llbracket g_j(z) \rrbracket \coloneqq \begin{cases*}
		    Q \llbracket 0 \rrbracket & $z\in \pi_0 \setminus \Ucal$ \\ 
            w_l(z) = \sum_{j=1}^Q \left\llbracket \kappa r_l^{\frac{Qk+1}{Q}} \eta\left(\frac{z - z_l}{r_l}\right) v_j\left(\frac{z - z_l}{r_l}\right)\right\rrbracket & $z\in B_{r_l}(z_l)$.
		\end{cases*}
	\end{equation}

We claim that $\Gbf_F$ is a $(C_0,2\alpha,R_0)$-almost minimizer for an appropriate choice of positive constants $C_0,R_0$ and $\kappa$. Let $\Bbf_r(x_0)\subset \R^4$ be such that $\Bbf_r(x_0)\cap \spt\Gbf_F\neq \emptyset$. To this end, we subdivide into three cases, based on $I\coloneqq \{ l : z_l\in B_{2r}(\pbf_{\pi_{0}}(x_0))\}$ and $I^{*} = \{ l : B_{r_{l}}(z_{l}) \cap B_{r}(\mathbf{p}_{\pi_0}(x_0)) \neq\emptyset\}$:
	\begin{enumerate}
		\item[(a)] $I=I^*=\emptyset$
		\item[(b)] $I=\emptyset$ and $I^* \neq \emptyset$
		\item[(c)] $I\neq\emptyset$
	\end{enumerate}
	In case (a), $\Gbf_{F} \restr \Bbf_{r}(x_{0})$ is the current $Q\llbracket \Bbf_r(x_0)\cap (\pi_0\times \{0\})\rrbracket$, which is clearly area minimizing in $\Bbf_{r}(x_{0})$. 
	
We emphasize a key distinction between the remaining cases:  in (b) we will need to reparameterize the sheets of $F$ over a possibly tilted plane {(relative to $\pi_0$)}, and in case (c) we will not need to reparameterize. This is because in case (c), the plane $\pi_0$ is sufficiently close to optimal, while in case (b) it may not be. The need to reparametrize in case (b) crucially requires us to know that the sheets of $F$ are disjoint single-valued graphs in the entirety of $\Bbf_r(x_0)$, which is due to the lack of branch points in $B_{2r}(\pbf_{\pi_0}(x_0))$, unlike in case (c).

Let us begin with case (c). Here, we will proceed to define an open subset $\Omega'\subset B_{2r}(\mathbf{p}_{\pi_0}(x_0))$ such that $|B_{2r}(\mathbf{p}_{\pi_0}(x_0))\setminus \Omega'| = 0$, and for which the hypotheses of Lemma \ref{l:almostmin} hold, directly for the functions $g_i$. Indeed, let 
\[
    \Omega' \coloneqq B_{2r}(\mathbf{p}_{\pi_0}(x_0)) \setminus \bigcup_l [z_l,z_l + r_l].
\]
In light of the introduction of a branch cut in each $\Omega'\cap B_{r_l}(z_l)$, combined with the pairwise disjointness of the supports of the functions $w_l$, a choice of the complex logarithm can be made to ensure that
\begin{equation}\label{e:rep}
    g_j(\zeta) = \kappa r_{l}^{\frac{Qk+1}{Q}} \eta \left( \frac{\zeta-z_{l}}{r_l} \right) \left( \frac{\zeta-z_{l}}{r_l} \right)^{\frac{(Qk+1)}{Q}}e^{\frac{j2\pi i}{Q}} \qquad \forall \zeta \in \Omega' \cap  B_{r_l}(z_l)
\end{equation}
and is therefore a $C^1$ function on $\Omega'$.

Now fix any index $l\in I$ and any $\zeta\in \Omega'\cap B_{r_l}(z_l)$. We have
\begin{align*}
    |\nabla g_j(\zeta)| &\lesssim_{Q,k} |\zeta - z_l|^{\frac{Qk+1-Q}{Q}} \eta \left(\frac{\zeta - z_l}{r_l}\right) + r_l^{-1}|\zeta-z_l|^{\frac{Qk+1}{Q}}\left|\nabla\eta\left(\frac{\zeta - z_l}{r_l}\right)\right| \\
    &= |g_j(\zeta)|^{\frac{Qk+1-Q}{Qk+1}}\left[\eta\left(\frac{\zeta - z_l}{r_l}\right)^{\frac{Q}{Qk+1}} + \left|\frac{\zeta-z_l}{r_l}\right|\left|\nabla\eta\left(\frac{\zeta - z_l}{r_l}\right)\right|\eta\left(\frac{\zeta-z_l}{r_l}\right)^{-\frac{Qk+1-Q}{Qk+1}}\right]
\end{align*}
Now, \eqref{e:etabounded} and $\supp(\eta)\subset \overline{B_{1/2}(0)}$ imply that
\[
    |\zeta||\nabla \eta(\zeta)|\eta(\zeta)^{-\frac{Qk+1-Q}{Qk+1}} \lesssim_{Q,k} 1 \qquad \forall \zeta \in \pi_0.
\]
This, combined with the fact that  $r_{l} \le \frac{1}{4}$ and $\|\eta\|_{C^0} \lesssim 1$ therefore yields the existence of a choice of $\kappa$ depending only on $Q$ and $k$ so that for every $\zeta \in \Omega' \cap B_{r_l}(z_l)$:
\begin{align}\label{e:theconstant}
    \nonumber |\nabla g_j(\zeta)| & \lesssim_{Q,k} |g_j(\zeta)|^{\frac{Qk+1-Q}{Qk+1}} \lesssim_{Q,k} \kappa r_l^{\frac{Qk+1-Q}{Q} - \frac{Qk+1-Q}{Qk+1}}|\zeta-z_l|^{\frac{Qk+1-Q}{Qk+1}} \\
    & \le \kappa |\zeta-z_l|^{\frac{Qk+1-Q}{Qk+1}} \le \frac{1}{4}.
\end{align}
We claim that the penultimate inequality in \eqref{e:theconstant} yields
\begin{equation} \label{e:pre3.1}
|\nabla g_{j}(\zeta)| \lesssim (\diam \Omega')^{\frac{Qk+1-Q}{Qk+1}} \qquad \forall ~\zeta \in \Omega'.
\end{equation}
Assuming \eqref{e:pre3.1} holds, one can readily check the hypotheses (i), (ii), and (iii) of Lemma \ref{l:almostmin} by the triangle inequality.

Let us now proceed to verify \eqref{e:pre3.1}. Note it is trivial whenever $\nabla g_{j}(\zeta) = 0$. So we assume $\zeta \in \Omega' \cap B_{r_l}(z_l)$ for some $l \in I \cup I^*$. If $l \in I$ then $z_l, \zeta \in B_{2r}(\mathbf{p}_{\pi_0}(x_0))$ so $|\zeta - z_l|  \le 4r = \diam \Omega'$, so \eqref{e:theconstant} implies \eqref{e:pre3.1}. On the other hand if $l \in I^{*} \setminus I$, because we are in case (c) there exists $l'\in I$ such that $z_{l'}\in B_{2r}(\mathbf{p}_{\pi_0}(x_0))$. Since for each $z_{i}$, $B_{2r_{i}}(z_{i}) \subset L_{i} \in \Wscr$ it follows that $|\zeta-z_l| \leq r_{l} \le |z_{l'} - \zeta| \le \diam \Omega'$. Thus, again \eqref{e:theconstant} implies \eqref{e:pre3.1} in this case as claimed.

We therefore apply Lemma \ref{l:almostmin} with this choice of $\Omega'$, to conclude that \eqref{e:Dir} holds for some constant $\bar C=\bar{C}(k,Q)>0$.\footnote{In particular, $\bar C$ is independent of the specific choice of $\Omega'$ (which here depends on the positioning of the disk $B_{2r}(\mathbf{p}_{\pi_0}(x_0))$).} Since $|B_{r}(\pbf_{\pi_0}(x_0) \setminus \Omega'| = 0$ and $\diam(\Omega')=2r$, \eqref{e:Dir} implies \eqref{e:dir}, so proceeding as in the proof of Proposition \ref{c:almostmin} from that point on, we conclude that
	$$
	\|\Gbf_F\|(\Bbf_r(x_0))\leq \|\Gbf_F + \partial S\|(\Bbf_r(x_0)) + C_0r^{2+2\alpha},
	$$
where $C_0= C_0(Q,k)>0$.  Note there is no restriction on the scale $r$ in this case.

We now turn our attention to case (b). In this case, $|\nabla g_{j}|$ could be large relative to $r^\alpha$ so the hypotheses of Lemma \ref{l:almostmin} cannot be satisfied {directly for $g_i$} and we instead must appeal to Proposition \ref{c:almostmin}. Up to relabelling the indices, choose $q\leq Q$ so that $g_1,\dots,g_q$ denote the functions as defined in \eqref{e:keytoentiretyofcasec} whose graphs intersect $\Bbf_r(x_0)$. In this case, we claim that the functions $g_{j}$ are single-valued $C^{k,\frac{1}{Q}}$ functions on the entirety of $ B_{2r}(\mathbf{p}_{\pi_0}(x_0))$ and further that \eqref{e:rep} holds with $\Omega' = B_{2r}(\pbf_{\pi_0}(x_0))$ and
     \begin{equation} \label{e:fullreg}
     \|g_j\|_{C^{1,\alpha}(\Omega')}\lesssim_{Q,k} 1.
     \end{equation}
     That $g_{j}$ is single-valued and that \eqref{e:rep} hold both follow from the assumption $I = \emptyset$ in case (b). Since \eqref{e:rep} holds, so does \eqref{e:theconstant}. Since $\Omega'$ is a ball, this guarantees that $\Lip(g_{j}) \le \frac{1}{4}$. Since $I = \emptyset$, observe that $B_{2r}(\pbf_{\pi_0}(x_0)) \cap B_{r_l}(z_l) \subset \Hbb_{l} \cap  B_{r_{l}}(z_{l})$, where $\Hbb_{l}$ is the halfspace through $z_l$ with normal in the direction of $\pbf_{\pi_{0}}(x_0) - z_{l}$. Thus, it follows from Lemma \ref{t:uniformity}, \eqref{e:rep}, and the fact that $z \mapsto z^{\frac{Qk+1}{Q}}$ is $C^{k,\frac{1}{Q}}(\Hbb)$ for any halfspace $\Hbb$ through the origin (with norms depending only on $k,Q$ and not the halfspace) that if $\zeta, \xi \in \Omega' \cap B_{r_{l}}(z_{l})$ then $|\nabla g_{j}(\zeta) - \nabla g_{j}(\xi)| \lesssim_{k,Q} |\zeta - \xi|^{\alpha}$. If $\zeta \in B_{r_{l}}(z_{l})$ and $\xi \in B_{r_{l'}}(z_{l'})$ for distinct $l, l'$ then $|\xi - \zeta| \ge r_{l} + r_{l'}$ since $B_{2r_{i}}(z_{i}) \subset L_{i}$ for all $L_{i} \in \Wscr$ with $\ell(L_{i}) \le 1$. In particular, recalling $\alpha = \frac{Qk+1-Q}{Qk+1}$ and applying \eqref{e:theconstant} yields
     $$
     | \nabla g_{j}(\xi) - \nabla g_{j}(\zeta)| \le |\nabla g_{j}(\xi)| + |\nabla g_{j}(\zeta)| \lesssim_{Q,k} r_{l}^{\alpha} + r_{l'}^{\alpha} \lesssim |\xi - \zeta|^{\alpha}.
     $$
     Together with the estimate when $\xi,\zeta$ are in the same ball and the fact that $g_{j} \equiv 0$ outside $\cup_{l} B_{r_l}(z_{l})$ it follows \eqref{e:fullreg} holds. As a consequence, $[\nabla g_{j}]_{C^{\alpha}(\Omega')} \lesssim 1$ so the only hypothesis of Lemma \ref{c:almostmin} left to check is \eqref{e:C1alpha.2}.

    Now, we check \eqref{e:C1alpha.2}. By \eqref{e:rep}, we have
    \begin{equation*}
        g_{j}(\zeta) - g_{j'}(\zeta) = g_{j}(\zeta) \left(1 - e^{\frac{2 \pi i (j'-j)}{Q}} \right)
    \end{equation*}
        for each $\zeta\in B_{2r}(\mathbf{p}_{\pi_0}(x_0))$. 
    Thus, following the computations leading to the first inequality in \eqref{e:theconstant} (except now for any point $z\in B_{2r}(\mathbf{p}_{\pi_0}(x_0))$), we deduce that for each $j\neq j'$,
	\[
	|\nabla g_j(\zeta)-\nabla g_{j'}(\zeta)| \lesssim_{Q,k} |g_j(\zeta) - g_{j'}(\zeta)|^{\frac{Qk+1-Q}{Qk+1}} \qquad \forall \zeta \in B_{2r}(\mathbf{p}_{\pi_0}(x_0)),
	\]
   It follows from Proposition \ref{c:almostmin} that in case (b) there exists $R_0=R_0(k,Q) >0$ such that whenever $r\leq R_0$,
	$$
	\|\Gbf_F\|(\Bbf_r(x_0))\leq \|\Gbf_F + \partial S\|(\Bbf_r(x_0)) + C_0r^{2+2\alpha},
	$$
	for some positive constant $C_{0}=C_0(k,Q)$. Since case (b) is the only case restricting $R_{0}$, we may take this choice of $R_0$ for the conclusion of the theorem. Choosing $C_{0}$ to be the largest constant from each of the above cases therefore completes all the necessary work to verify that $\Gbf_{F}$ is a $(C_{0}, 2\alpha, R_{0})$-almost minimizer in $\R^{4}$.

    It remains to check that for $T = \Gbf_{F}$ we have $K \subset \Sing(T) = \Ffrak_{Q}(T)$. The fact that $\Sing(T) \supset K$ follows since the singular set is always closed and $\overline{ \cup_{l} \{z_{l}\}} \supset K$, while each $z_l$ is a branching singularity of $\Gbf_{F}$. To see that $\Sing(T)=\Ffrak_Q(T)$, we argue as follows. First of all, clearly $\{z_l\}_{l\in\N}\subset \Ffrak_Q(T)$. It therefore remains to check that $K\subset \Ffrak_Q(T)$. Since 
    $$
    \| w_{l}\|_{L^\infty} \lesssim_{k,Q} r_{l}^{\frac{Qk+1}{Q}} \simeq  \dist( \supp (w_{l}), K)^{\frac{Qk+1}{Q}},
    $$ and the $w_{l}$ have disjoint support, it follows that for each $x\in K$,
    $$
    \dist( \Bbf_{r}(x) \cap \spt(T) , \Bbf_{r}(x) \cap \pi_{0}) \lesssim r^{\frac{Qk+1}{Q}}.
    $$
    This, together with the $Q$-valued graphicality of $T$, in turn yields that for any such $x$, the rescalings\footnote{Here $\iota_{x,r}(y)\coloneqq \frac{y-x}{r}$ and we let $(\iota_{x,r})_\sharp T$ denote the pushforward current under the map $\iota_{x,r}$.} $T_{x,r}\coloneqq (\iota_{x,r})_\sharp T$ satisfy
    \[
        T_{x,r}\mres \Bbf_1 \toweakstar Q\llbracket \pi_0 \cap \Bbf_1 \rrbracket \qquad 
        \text{along any subsequence as $r\todown 0$.}
    \]
\end{proof}

\appendix

\section{Reparameterizations}
Here, we introduce a key reparameterization result which is used frequently throughout this article. We first recall the following elementary fact. For an affine function $A$, note that $\frac{1}{\sqrt{1 + \Lip(A)^{2}}}$ and $\frac{\Lip(A)}{\sqrt{1 + \Lip(A)^{2}}}$ are respectively the cosine and sine of the (maximal one-dimensional) angle between the domain of $A$ and the graph of $A$.

\begin{proposition} \label{p:step1}
    Fix $r,\sigma,\Lambda>0$ and set $\tau = (1 + \Lambda^{2})^{-1/2}$.  Let $\pi_0$ be an $m$-dimensional plane in $\R^{m+n}$, $x_0=(x', \bar{x}) \in \pi_0\times \pi_0^\perp \equiv \R^{m+n}$,  $A : \pi_0 \to \pi_0^{\perp}$ a linear function with graph parallel to some $\varpi \in G(m,m+n)$, and $\delta > 0$ satisfy
    \begin{equation} \label{e:closeenough}
            \delta + \frac{ \Lip(A) \sigma}{\sqrt{1+\Lip(A)^{2}}}  \le  \tau.
    \end{equation} 
    Suppose that $f \in \Lip \left( B_{\delta^{-1} r}(x',\pi_0); \pi_0^\perp \right)$. To each $x \in B_{\delta^{-1} r}(x')$ define $y = y(x) \in \varpi$ by $y = \pbf_{\varpi}(x, f(x))$. In particular,  $y' = y(x')$. 
    \begin{enumerate}
        \item If 
        $\Lip(f - A) \le \Lambda < \infty,$
        then there exist $g : B_{\tau \delta^{-1} r}(y') \to \varpi^{\perp}$ so that $\gr(g)= \gr(f) \cap \Cbf_{\tau \delta^{-1} r}(y',\varpi)$.

        \item If $| f(x') - \bar{x}| \le \sigma \delta^{-1} r$ then it follows that $B_{r}(\pbf_{\varpi}(x_0)) \subset B_{\tau \delta^{-1} r}(y').$
        
        \item  If $f \in C^{1,\alpha}(B_{\delta^{-1} r}(x'),\pi_0^\perp)$, {there exists $R_{2} = R_{2}([\nabla f]_{C^{\alpha}(B_{\delta^{-1} r}(x'))},m,n,\Lambda)>0$ such that if} $\delta^{-1} r \le R_{2}$ then $g \in C^{1,\alpha}(B_{\tau \delta^{-1} r}(y'),\varpi^\perp)$ and there exists a constant $C = C(\Lip(f),m,n)$ so that
        \begin{align} \label{e:hseminorm}
        [\nabla g]_{C^{\alpha}(B_{\tau \delta^{-1} r}(y'))} &\le C [\nabla f]_{C^{\alpha}(B_{\delta^{-1} r}(x'))} \equiv [\nabla (f-A)]_{C^{\alpha}(B_{\delta^{-1} r}(x'))}.
        \end{align}
        \item If $f_{1}, f_{2} \in \Lip \left( B_{\delta^{-1} r}(x') ; \pi_0^\perp \right)$ satisfy the same hypotheses as $f$ in (1) and (2), if $\delta$ is as in (2) and $g_{1}, g_{2}$ are the corresponding functions from (1) for $f_{1}, f_{2}$ respectively then 
        $$
        B_{r}(\pbf_\varpi(x_0)) \subset B_{\tau\delta^{-1}r}(\pbf_\varpi(x',f_1(x')))\cap B_{\tau\delta^{-1}r}(\pbf_\varpi(x',f_2(x')))
        $$ 
        and for any $z\in B_{r}(\pbf_\varpi(x_0))$, we have
        \begin{align}
\nonumber        |\nabla g_{1}(z) - \nabla g_{2}(z)|&\le \frac{\Lip(A)^\alpha}{(1 + \Lip(A)^{2})^{\frac{\alpha}{2}}}  \min_{i=1,2}[\nabla f_{i}]_{C^{\alpha}(B_{\delta^{-1} r}(x'))} |f_1(y_1^{-1}(z)) - f_2(y_2^{-1}(z))|^\alpha \\
\label{e:relgrad}        &\qquad+ \|\nabla f_1 - \nabla f_2\|_{C^0(B_{\delta^{-1} r}(x'))}.
        \end{align}
        \end{enumerate}
\end{proposition}

\begin{proof}
    We may without loss of generality assume that $\pi_0 = \R^m \equiv \R^m\times \{0\}^{n}$. To prove (1), note that $\Lip(f-A) \le \Lambda$ implies that for $z_{1}, z_{2} \in \gr(f)$,
    $$
    |z_{1} - z_{2}|^{2} = |\pbf_{\varpi}(z_1) - \pbf_{\varpi}(z_2)|^{2} + |\pbf_{\varpi^{\perp}}(z_{1})- \pbf_{\varpi^{\perp}}(z_{2})|^{2} \le (1+\Lambda^{2})|\pbf_{\varpi}(z_1) - \pbf_{\varpi}(z_2)|^{2},
    $$
    so $\pbf_{\varpi}|_{\gr(g)}$ is invertible on its image. We claim this image contains $B_{\tau \delta^{-1} r}(\pbf_{\varpi}(x',f(x')))$. Indeed, let $\pi$ denote the translate of $\varpi$ passing through $(x', f(x'))$. Then, $\Lip(f - A) \le \Lambda$ implies
    $$
        \gr (f) \cap \Bbf_{\delta^{-1} r}((x',f(x'))) \subset \left\{ z \in \Bbf_{\delta^{-1} r}((x', f(x'))) : \dist(z, \pi) \le \Lambda \delta^{-1} r \right\}.
    $$
    Hence, the graphicality of $f$ ensures $\pbf_{\varpi}(\gr(f) \cap \Bbf_{\delta^{-1} r}(x',f(x'))$ contains the $\pbf_{\varpi}$-image of the disc
    $$
        \{ z \in \Bbf_{\delta^{-1} r}((x',f(x'))) : \dist(z, \pi) = \Lambda r \}.
    $$
    By the Pythagorean theorem, this is precisely the disk $B_{\tau \delta^{-1} r}(\pbf_{\varpi}(x',f(x')))$. 

    Meanwhile, (2) follows since $|f(x') - \bar{x}| \le \sigma \delta^{-1} r$ implies
    \begin{align*}
    |\pbf_{\varpi}( x', f(x')) & - \pbf_{\varpi}(x_0)| \le \frac{\Lip(A)}{\sqrt{1+\Lip(A)^{2}}} |(x', f(x')) - x_0| \\
    & = \frac{\Lip(A)}{\sqrt{1+\Lip(A)^{2}}} |f(x') - \bar{x}| \le \frac{\sigma r \Lip(A)}{\sqrt{1+ \Lip(A)^{2}}}.
    \end{align*}
    In particular, if $y \in B_{r}(\pbf_{\varpi}(x_0))$ then 
    $$
        |y - \pbf_{\varpi}( x', f(x'))| \le |y- \pbf_{\varpi}(x_0)| +  \frac{\sigma \delta^{-1} r \Lip(A)}{\sqrt{1+ \Lip(A)^{2}}} \le r + \frac{\sigma \delta^{-1} r\Lip(A)}{\sqrt{1+\Lip(A)^{2}}}
    $$
    so that \eqref{e:closeenough} implies $y \in B_{\tau \delta^{-1} r}(y')$ as desired.

        To prove (3), consider the map $x \mapsto y(x)$ and note it has explicit form
        $$
            y(x) = \pbf_{\varpi}( x,f(x)).
        $$

        Let $\{e_{i}\}_{i=1}^{m+n}$ be an orthonormal basis for $\R^{m + n}$ adjusted to the decomposition $\R^{m} \times \R^{n}$.
        Since $A$ graphs $\varpi$, it follows that for $i=1, \dots, m$
        $$
        \xi_{i} \defeq \frac{e_{i} + A e_{i}}{|e_{i} + A e_{i}|}
        $$
        is an orthonormal basis for the embedding of $\varpi$ in $\R^{m+n}$. Choose $\xi_{m+j}$ for $j=1, \dots, n$ so that $\{\xi_{i}\}_{i=1}^{m+n}$ is an orthonormal basis for $\R^{m+n}$.

        With respect to these bases, it follows that {$(\nabla F)^{i}_{j} = \nabla (F \cdot \xi_{i}) \cdot e_{j}= |(e_{i} + A(e_{i})| \delta^i_j$ for $i,j =1, \dots, m$}, where $\delta^i_{j}$ is the Kroenecker delta. In particular, with respect to these coordinates, $(\nabla F)_{1 \le i,j \le m}$ is a diagonal matrix with entries at least $1$. On the other hand, for $i,j = 1, \dots, m$,
        $$
            |\nabla (y-F)^{i}_{j}| = |\nabla( (y-F) \cdot \xi_{i} ) \cdot e_{j} |= |\xi_{i}^{t}  \left( \nabla \left( f - A \right) \right) e_{j}|  \le [\nabla f]_{C^{\alpha}(B_{\delta^{-1} r}(x'))} (\delta^{-1} r)^{\alpha}.
        $$
        Therefore, for $R_{2}$ small enough depending on $\delta, m,n$, and $[\nabla f]_{C^{\alpha}(B_{\delta^{-1} r}(x'))}$ and hence depending on $\Lambda, m,n$, and $[\nabla f]_{C^{\alpha}(B_{\delta^{-1} r}(x'))}$ all eigenvalues of\footnote{At the risk of abusing notation, we are henceforth identifying the $m\times n$ matrix $\nabla y$ with the square $m\times m$ matrix consisting of its non-zero entries.} $(\nabla y) \equiv (\nabla y)_{1\leq i,j \leq m}$ are at least $\frac{1}{2}$, so we may apply the inverse function theorem. Let $y^{-1}$ denote its inverse, which satisfies
        \begin{equation}\label{e:invft}
            \nabla y^{-1}(z) = (\nabla y( y^{-1}(z)))^{-t} 
        \end{equation}

        Hence, we can bound 
        $[ \nabla y^{-1}]_{C^{\alpha}(B_{\tau \delta^{-1} r}(y'))}$ by $[\nabla y]_{C^{\alpha}(B_{\delta^{-1} r}(x'))}$. Indeed, since $(\nabla y)^{-1}$ is a product of $\det( \nabla y)^{-1}$ and an alternating sum of determinants of the principle minors of $\nabla y$ it follows from the fact that $\det \in \Lip(\R^{m \times m} ; \R)$, $\det( \nabla y) \ge 2^{-m}$, and $t \mapsto t^{-1} \in \Lip( [2^{-m},\infty), \R)$ that $[\nabla y^{-1}]_{C^{\alpha}(B_{\tau \delta^{-1} r}(y'))}$ is bounded by $[\nabla y]_{C^{\alpha}(B_{\delta^{-1} r}(x'))}$ with only dimensional dependencies. We may analogously bound $\|\nabla y^{-1}\|_{C^0(B_{\tau \delta^{-1}r}(y'))}$ by $\|\nabla y\|_{C^0(B_{\delta^{-1}r}(x'))}$.

        Observe that for $\Phi(x)\coloneqq \pbf_{\varpi}^\perp(x,f(x))$, we have $g(z) = \Phi\circ y^{-1}(z)$. The chain rule, together with the regularity of $f$ implies that $g\in C^{1,\alpha}(B_{\delta^{-1} \tau r}(y');\varpi^\perp)$. It remains to check the H\"{o}lder estimate. For this, we again use the chain rule to compute 
        \begin{align*}
            [\nabla g]_{C^{\alpha}} &= [\nabla(\Phi\circ y^{-1})]_{C^{\alpha}} \\
            &= [(\nabla y^{-1})^t (\nabla \Phi\circ y^{-1})]_{C^{\alpha}} \\
            &\leq \|\nabla \Phi\circ y^{-1}\|_{C^0} [\nabla y^{-1}]_{C^{\alpha}} + [\nabla \Phi\circ y^{-1}]_{C^{\alpha}} \| \nabla y^{-1}\|_{C^0},
        \end{align*}
        where the domain for all the norms and seminorms is $B_{\tau \delta^{-1} r}(y')$. Combining this with the above bounds on $[ \nabla y^{-1}]_{C^{\alpha}(B_{\tau \delta^{-1} r}(y'))}$ and $\|\nabla y^{-1}\|_{C^0(B_{\tau \delta^{-1}r}(y'))}$, the conclusion follows.
        
        Finally, we prove (4). Letting $y_i(x) \coloneqq (x,f_i(x))$ for $i=1,2$, and let $\Phi_i$ be the respective functions such that $g_i=\Phi_i\circ y_i^{-1}$,  and suppose $z \in B_{r}(\pbf_{\varpi}(x_0))$. Without loss of generality, suppose that $[\nabla f_1]_{C^{\alpha}(B_{\delta^{-1} r}(x'))}\leq [\nabla f_2]_{C^{\alpha}(B_{\delta^{-1} r}(x'))}$. We have
        
        \begin{align*}
            |\nabla g_1(z) - \nabla g_2(z)| &= |\nabla \Phi_1(y_1^{-1}(z)) - \nabla \Phi_2(y_2^{-1}(z))| \\
            &\le|\nabla \Phi_1(y_1^{-1}(z)) - \nabla \Phi_1(y_2^{-1}(z))| + |\nabla \Phi_1(y_2^{-1}(z)) - \nabla \Phi_2(y_2^{-1}(z))| \\
            & \le [\nabla \Phi_{1}]_{C^{\alpha}} |y_{1}^{-1}(z) - y_{2}^{-1}(z)|^\alpha + \|\nabla \Phi_1 - \nabla \Phi_2\|_{C^0(B_{\delta^{-1} r}(x'))} \\
            &\le \frac{\Lip(A)^\alpha}{(1 + \Lip(A)^{2})^{\frac{\alpha}{2}}} [\nabla \Phi_{1}]_{C^{\alpha}(B_{\delta^{-1} r}(x'))} |f_1(y_1^{-1}(z)) - f_2(y_2^{-1}(z))|^\alpha \\
            &\qquad+ \|\nabla \Phi_1 - \nabla \Phi_2\|_{C^0(B_{\delta^{-1} r}(x'))} \\
            &\le \frac{\Lip(A)^\alpha}{(1 + \Lip(A)^{2})^{\frac{\alpha}{2}}}  [\nabla f_1]_{C^{\alpha}(B_{\delta^{-1} r}(x'))} |f_1(y_1^{-1}(z)) - f_2(y_2^{-1}(z))|^\alpha \\
            &\qquad+ \|\nabla f_1 - \nabla f_2\|_{C^0(B_{\delta^{-1} r}(x'))}.
        \end{align*}
         
\end{proof}

\bibliographystyle{alpha}
\bibdata{references}
\bibliography{references}

\end{document}